\documentclass[oneside,a4paper,12pt,centertags]{amsart}
\usepackage{amssymb, amsthm, amscd, latexsym, amsmath, verbatim}
\usepackage{vmargin}
\usepackage[bookmarks=true]{hyperref} 

\newtheorem{theorem}{Theorem}[section]
\newtheorem{lemma}[theorem]{Lemma}

\theoremstyle{definition}

\theoremstyle{remark}
\newtheorem{remark}[theorem]{Remark}

\numberwithin{equation}{section}

\newcommand{\R}{\mathbb{R}}
\newcommand{\C}{\mathbb{C}}
\newcommand{\N}{\mathbb{N}}
\newcommand{\cL}{\mathcal{L}}
\newcommand{\cV}{\mathcal{V}}
\newcommand{\cX}{\mathcal{X}}
\newcommand{\Dom}{\mathsf{Dom}}
\newcommand{\Ran}{\mathsf{R}}
\newcommand{\Nul}{\mathsf{N}}
\newcommand{\Proj}{\mathsf{P}}
\newcommand{\sppt}{\mathop{\mathrm{sppt}}}
\newcommand{\dist}{\mathop{\mathrm{dist}}}
\newcommand{\Lip}{\mathop{\mathrm{Lip}}}
\newcommand{\Div}{\mathop{\mathrm{div}}}
\newcommand{\re}{\mathop{\mathrm{Re}}}
\newcommand{\im}{\mathop{\mathrm{Im}}}

\begin{document}
\title[Finite Propagation Speed for First Order Systems]{Finite Propagation Speed for First Order Systems and Huygens' Principle for Hyperbolic Equations}
\author[McIntosh]{Alan McIntosh}
\author[Morris]{Andrew J. Morris}

\address{Alan McIntosh \\ Centre for Mathematics and its Applications\\ Australian National University\\ Canberra\\ ACT 0200\\ Australia\vspace{-.1cm}}
\email{alan.mcintosh@anu.edu.au}

\address{Andrew J. Morris\\ Department of Mathematics\\ University of Missouri\\ Columbia \\ MO 65211\\ USA\vspace{-.1cm}}
\email{morrisaj@missouri.edu}

\subjclass[2010]{Primary: 35F35, 35L20; Secondary: 47D06}
\keywords{Finite propagation speed, first order systems, $C_0$~groups, Huygens' principle, hyperbolic equations.}

\date{24 January 2012}

\begin{abstract}
We prove that strongly continuous groups generated by first order systems on Riemannian manifolds have finite propagation speed. Our procedure  provides a new direct proof for self-adjoint systems, and allows an  extension to operators on metric measure spaces. As an application, we present a new approach to the  weak Huygens' principle for second order hyperbolic equations.
\end{abstract}

\maketitle

\section{Introduction}
For a self-adjoint first order differential operator $D$ acting on a space $L^2(\mathcal{V})$, where  $\mathcal{V}$  is a vector bundle over a complete Riemannian manifold $M$, it is known that the unitary operators $e^{itD}$  act with finite propagation speed. Indeed, if the principal symbol satisfies 
\begin{equation}\label{symbol} |\sigma_D(x,\xi)|\leq \kappa |\xi| \qquad\forall x\in M,\ \forall \xi\in T^*_xM \end{equation}
for some positive number $\kappa$, and if $\sppt (u) \subset K\subset M$, then 
$$\sppt (e^{itD}u) \subset K_{\kappa |t|}: = \{x\in M \,;\, \dist(x,K)\leq \kappa |t|\}\qquad \forall t\in\R,$$
 that is, $e^{itD}$ has {\it finite propagation speed }$\kappa_D\leq\kappa$.
See, for instance, Propositions~10.2.11 and~10.3.1 in~\cite{HigsonRoe2000}.

Noting that  $\sigma_D(x,\nabla\eta(x))u(x)= [D,\eta I]u(x)$ for all bounded real-valued $C^1$ functions  $\eta$ on $M$,    we see that  condition \eqref{symbol} implies that the commutator $[\eta I,D]$ is a multiplication operator which satisfies the bound
\begin{equation}\label{comm}\|[\eta I, D]u\|_2 \leq \kappa \|\nabla \eta\|_\infty \|u\|_2\quad \forall u\in \Dom(D)\subset\Dom(D\eta),\end{equation}
where  $\Dom(D)$ denotes the domain of   $D$ and $[\eta I, D]u=\eta Du-D(\eta u)$.

On stating the result in terms of commutators, we can remove the differentiability assumptions altogether and consider operators $D$ defined on metric measure spaces instead. Our aim is to also weaken the self-adjointness condition on $D$ to the requirement that  $iD$ generates a $C_0$  group $(e^{itD})_{t\in\R}$   with 
\begin{equation}\label{group}\|e^{itD}u\|_2\leq c e^{\omega |t|}\|u\|_2\quad\forall t\in \R,\ \forall u \in L^2(\cV)
\end{equation}
 for some $c\geq1$ and $\omega\geq 0$.  (When $D$ is self-adjoint, condition \eqref{group} holds with $c=1$ and $\omega = 0$.)
 At the same time we could  replace  the $L^2$ space by an  $L^p$ space if we wished. This requires a new proof of finite propagation speed. Let us state the result here for Riemannian manifolds.

\begin{theorem} 
\label{Thm: MainThm1}
Let  $D$ be a first order differential operator which acts  on a space $L^2(\mathcal{V})$, where  $\mathcal{V}$  is a complex vector bundle  with a Hermitian metric, over a separable  Riemannian manifold $M$. Suppose that $iD$ generates a $C_0$ group $(e^{itD})_{t\in\R}$ satisfying \eqref{group} and that the commutators of $D$ with bounded real-valued $C^\infty$ functions $\eta$ on~$M$ satisfy \eqref{comm}. Then the group  $(e^{itD})_{t\in\R}$  has finite propagation speed    $\kappa_D\leq c\kappa $. 
\end{theorem}

In particular, the operator $D$ acting on $L^2(\mathcal{V})$ could denote a first order system acting on $L^2(\R^n,\,\C^N)$ for some positive integers $n$ and $N$, or on $L^2(\Omega,\,\C^N)$ where $\Omega$ is an open subset of $\R^n$. We remark that the constants in  \eqref{comm} and  \eqref{group} could be with respect to another norm on $L^2(\cV)$ equivalent to the standard one. We shall return to this point.

In Section~\ref{sec:main}, we prove our main result, Theorem~\ref{Thm: MainThm2}, which is a generalisation of Theorem~\ref{Thm: MainThm1} to metric measure spaces.  The proof utilises a higher-commutator technique introduced by McIntosh and Nahmod in Section~2 of \cite{McIntoshNahmod2000}, and  used to derive off-diagonal estimates, otherwise known as Davies--Gaffney estimates, by Axelsson, Keith and McIntosh in Proposition~5.2 of~\cite{AKMc}, and by Carbonaro, McIntosh and Morris in Lemma~5.3 of~\cite{CMcM}. The proof also simplifies the argument based on energy estimates that is known for self-adjoint operators.  

A weak Huygens' principle for second order hyperbolic equations on $\R^n$ is proved as an application in Section~\ref{sec:app}. Homogeneous and inhomogeneous hyperbolic equations are treated separately. The homogeneous version in Theorem~\ref{hom} only requires the finite propagation speed result for self-adjoint first order systems. The inhomogeneous version in Theorem~\ref{inhom}, however, requires the generality of Theorem~\ref{Thm: MainThm2}. These results are achieved by introducing a first order elliptic system $BD$, where $D$ is a first order constant coefficient system, and $B$ is a multiplication operator, such that the second order hyperbolic equation contains a component of the system $(BD)^2$. This approach is motivated by the work of Auscher, McIntosh and Nahmod~\cite{AMcN1997}, and  of Axelsson, Keith and McIntosh~\cite{AKMc,AKMc2}, in which the solution of the Kato square root problem for second order elliptic operators is reduced to proving quadratic estimates for related first order elliptic systems. See also the survey by Auscher, Axelsson and McIntosh~\cite{AAMc}.
 
\section{Notation}\label{sec:notation}
	Henceforth, $M$ denotes a {\it metric measure space} with a metric $d(x,y)$ and a $\sigma$-finite Borel measure $\mu$. In particular, $M$ could be a Riemannian manifold as in the Introduction, or an open subset of $\R^n$ with Euclidean distance and Lebesgue measure. If $K, \tilde K\subset M$ and $x\in M$, then $d(x,K):=\inf\{d(x,y)\,;\,y\in K\}$ and $d(K,\tilde K): =\inf\{d(x,y)\,;\,x\in K, y\in \tilde K\} $. We define, for $\tau>0$, $K_\tau:=\{x\in M\,;\,d(x,K)\leq\tau\}$.  By $\Lip(M)$ we mean the space of all bounded real-valued functions $\eta$ on $M$ with finite Lipschitz norm $$\|\eta\|_{\Lip} = \sup_{x\neq y}\tfrac{|\eta(x)-\eta(y)|}{d(x,y)}\ .$$ Whenever $K\subset M$ and $\alpha>0$, the real-valued function $\eta_{K,\alpha}$ defined by
\begin{equation}\label{alpha} \eta_{K,\alpha}(x):=\max\{1-\alpha d(x,K)\,,\,0\}
\end{equation}
 belongs to $\Lip(M)$ with $\|\eta_{K,\alpha}\|_{\Lip}\leq \alpha$, and  $\sppt(\eta_{K,\alpha})\subset
 K_{1/\alpha}$.

A {\it vector bundle} $\cV$ over $M$ refers to a  complex vector bundle $\pi:\cV\to M$  equipped with a Hermitian metric $\langle \cdot , \cdot \rangle_x$ that depends continuously on $x\in M$. The examples to be presented in Section \ref{sec:special} will be  trivial bundles $M\times \C^m$ with inner product $\langle\zeta,\xi\rangle_x = \sum_j \zeta_j\overline\xi_j$. For every vector bundle $\cV$, there are naturally defined Banach  spaces $L^p(\cV)$, $1\leq p\leq \infty$, of measurable sections. In particular, $L^2(\mathcal{V})$ denotes the Hilbert space of square integrable  sections of $\mathcal{V}$ with the inner product $( u\,, v):= \int_M \langle u(x)\,, v(x) \rangle_x\, d\mu(x)$. In the case of the trivial bundle $M\times \C^m$, these are denoted as usual by $L^p(M,\C^m)$. 

The Banach algebra of all bounded linear operators on a Banach space $\cX$ is denoted by  $\cL(\cX)$.
Given $A\in L^\infty(M,\cL(\C^m))$, the same symbol $A$ is also used to denote the
 {\it multiplication operator} on $L^p(M,\C^m)$ defined by $u\mapsto A u$. Note that $\|A u\|_p\leq\|A\|_{\infty}\|u\|_p$.  Multiplication operators on $L^p(\cV)$ are defined in the natural way.  For any function $\eta\in \Lip(M)$, the multiplication operator $\eta I:L^p(\mathcal{V})\rightarrow L^p(\mathcal{V})$ is defined by $(\eta I)u(x):=\eta(x)u(x)$ for all $u\in L^p(\mathcal{V})$ and $\mu$-almost all $x\in M$. This is a multiplication operator by virtue of the facts that $\eta$ is bounded and continuous, and $\mu$ is a Borel measure.
The {\it commutator} $[A,T]$ of a multiplication operator $A$ with a (possibly unbounded) operator $T$ in $L^p(\cV)$ with domain $\Dom(T)$ is defined by $[A,T]u=ATu - TAu$ provided $u\in\Dom(T)\cap\Dom(TA)$.

Given an operator $D$ in $L^p(\cV)$ ($1\leq p<\infty$), we say that $iD$ {\it generates a} $C_0$ {\it group} $(V(t))_{t\in\R}$ provided $t\mapsto V(t)$ is a strongly continuous mapping from $\R$ to $\cL(L^p(\cV))$ with $V(s+t)=V(s)V(t)$, $V(0)=I$ (the identity map on $L^p(\cV)$) and $\frac{d}{dt}V(t)u|_{t=0}=iDu$ for all $u\in \Dom(D)=\{u\in L^p(\cV)\,; \frac{d}{dt}V(t)u|_{t=0} \text{ exists in } L^p(\cV) \} $. We write $V(t) = e^{itD}$. Such a group automatically has dense domain $\Dom(D)$, and satisfies an estimate of the form $\|e^{itD}\| \leq c e^{{\omega} |t|}$ for some $c\geq 1$ and $\omega\geq 0$. An introduction to the theory of strongly continuous groups can be found in, for instance, \cite{Kato} or \cite{EngelNagel2000}. We remark that, when $D$ is self-adjoint in $L^2(\cV)$, Stone's Theorem guarantees that the operators $e^{itD}$ are unitary, so $iD$ generates a $C_0$ group with $c=1$ and $\omega=0$.

The group $(e^{itD})_{t\in\R}$ is said to have \textit{finite propagation speed} when there exists a finite constant $\kappa\geq 0$, such that for all  $u\in L^p(\mathcal{V})$ satisfying $\sppt(u)\subset K\subset M$, and all $t\in \R$, it holds that $\sppt(e^{itD}u)\subset K_{\kappa|t|}$.
The \textit{propagation speed} $\kappa_D$ is  defined to be the least such $\kappa$.

\section{The Main Result}\label{sec:main}

The following theorem is the main result of the paper. Theorem~\ref{Thm: MainThm1} is proved as a special case at the end of the section.

\begin{theorem}\label{Thm: MainThm2}  Let $\mathcal{V}$ denote a  complex vector bundle over a metric measure space~$M$ and let $1\leq  p < \infty$.
 Suppose that $D:\Dom(D)\subset L^p(\mathcal{V})\rightarrow L^p(\mathcal{V})$ is a linear operator with the following properties:
\begin{enumerate}
\item\label{1} there exist finite constants $c\geq 1$ and $\omega\geq 0$ such that $iD$ generates a $C_0$ group $(e^{itD})_{t\in\R}$ in $L^p(\mathcal{V})$ with $\|e^{itD}u\|_p \leq c e^{{\omega} |t|}\|u\|_p\ \forall t\in\R, \forall u\in L^p(\cV)$;
\item\label{2} there exists a finite constant $\kappa>0$ such that for all $\eta\in \Lip(M)$, one has $\eta u \in \Dom(D)$ and $\|[\eta I,D]u\|_p \leq \kappa \|\eta\|_{\Lip}\|u\|_p$ and $[\eta I,[\eta I, D]]u=0$ for all $u\in \Dom(D)$.
\end{enumerate}
Then the group $(e^{itD})_{t\in\mathbb{R}}$  has finite propagation speed $\kappa_{D} \leq c\kappa$.
\end{theorem}

\begin{remark}
As the commutator $[\eta I,D]$ is bounded on the dense domain $\Dom(D)$, it extends uniquely to an operator (denoted by the same symbol)  $[\eta I,D]\in \cL(L^p(\cV))$ with the same bound.
\end{remark}

\begin{remark}\label{othernorm} The theorem remains true when the norm on $L^p(\cV)$ is replaced by another equivalent norm. This is used  later in Case II of Section \ref{sec:special}.
\end{remark}

\begin{remark} The results remain true in Bochner spaces $L^2(\cV)$ when the fibres of $\cV$ are infinite dimensional Hilbert spaces.
\end{remark}

For completeness, we prove a known formula for the  commutator $[\eta I,e^{itD}]$.

\begin{lemma}\label{Lem: CommForm}
Under the hypotheses of Theorem~\ref{Thm: MainThm2}, the following holds:
\[
[\eta I,e^{itD}]u = it \int_0^1 e^{istD} [\eta I,D] e^{i(1-s)tD}u\, d s \qquad\forall t\in\R, 
\forall u\in L^p(\cV).\]
\end{lemma}

\begin{proof}
It suffices to verify the expression when $u\in\Dom(D)$. The property that $(e^{itD})_{t\in\R}$ is a $C_0$ group then guarantees that $e^{itD} u\in\Dom(D)$ with derivative $\tfrac d{dt}(e^{itD}u)=iDe^{itD}u=ie^{itD}Du$ for all $t\in\R$. The property that $\eta u \in \Dom(D)$ then implies that $e^{istD}(\eta I) e^{i(1-s)tD} u$ is differentiable with respect to $s$, with 
\begin{align*}
\frac d {ds}\Big(e^{is tD}(\eta I)  e^{i(1-s)tD} u\Big)
&= \Big(e^{i(\cdot)tD} (\eta I) e^{i(1-s)tD} u\Big)^\prime(s) + e^{istD} \Big((\eta I) e^{i(1-\,\cdot)tD} u\Big)^\prime(s)\\
&= -ite^{istD} [\eta I,D] e^{i(1-s)tD}u
\end{align*}
for all $s\in\R$. This version of the chain rule can be found in, for instance, Lemma~B.16 in~\cite{EngelNagel2000}. Using the fundamental theorem of calculus, we then have
\[ [\eta I,e^{itD}]u  = -\int_0^1\frac d {ds}\Big(e^{is tD}(\eta I)  e^{i(1-s)tD} u\Big)\, d s
= it \int_0^1 e^{istD} [\eta I,D] e^{i(1-s)tD}u\, d s\]
as required.
\end{proof}

\begin{proof}[Proof of Theorem \ref{Thm: MainThm2}]
Given $t\in\R$ and $u\in L^p(\cV)$ with $\sppt(u)\subset K \subset M$, our aim is to prove that $\sppt(e^{itD}u)\subset K_{c\kappa|t|}$. To do this, it suffices to prove that $(e^{itD}u,v)=0$ for all $v\in L^{p'}(\cV)$ with $d(\sppt(v), K)> c\kappa|t|$ (where $p'=\frac p {p-1}$). Let us fix $K, t, u, v$, and choose $\alpha>0$ such that  $c\kappa|t|<1/\alpha< d(\sppt(v), K)$. On defining the cut-off function $\eta:=\eta_{K,\alpha}\in\Lip(M)$ as  in \eqref{alpha}, we have $\eta u = u$, $\eta v = 0$ and $c\kappa|t|\|\eta\|_{\Lip}\leq c\kappa|t|\alpha <1$.

To simplify the computations, set $\delta:\cL(L^p(\cV))\to\cL(L^p(\cV))$ to be the {\it derivation} defined by $\delta(S)=[\eta I,S]$ for all $S\in\cL(L^p(\cV))$,  and adopt the convention that $\delta^0(S):=S$. We see that
\begin{equation}\label{easily}
(e^{itD}u,v)=(e^{itD}\eta^nu,v)=-( \delta(e^{itD})\eta^{(n-1)} u,v ) =\!\dots\!=
(-1)^n( \delta^n(e^{itD}) u,v )
\end{equation}
for all $n\in\N$. The derivation formula $\delta(ST)=\delta(S)T+S\delta(T)$ is readily verified for any $S,T\in\cL(L^p(\cV))$.
 Using Lemma~\ref{Lem: CommForm} and the fact, given by property~\eqref{2} of~$D$, that $\delta([\eta I, D])=[\eta I, [\eta I, D]]=0$, we then obtain
\begin{equation}\label{eq: ind.In}
\delta^{m+1}(e^{itD})u = it\int_0^1 \sum_{k=0}^{m} {\textstyle \binom{m}{k}} \delta^{m-k}(e^{istD}) [\eta I, D] \delta^k(e^{i(1-s)tD})u\, d{s}
\end{equation}
for all $m\in\N_0:=\N\cup\{0\}$, where the binomial coefficient $\binom{m}{k}:=\frac{m!}{k!(m-k)!}$. We now prove by induction that
\begin{equation}\label{eq: ind.NormIn}
\|\delta^n(e^{itD})\| \leq  (c\,|t|\,\|[\eta I,D]\|)^n ce^{{\omega}|t|}
\end{equation}
for all $n\in\N_0$. For $n=0$, this is given by property~\eqref{1} of~$D$. Now let $m\in\N$ and suppose that \eqref{eq: ind.NormIn} holds for all integers $n\leq m$. We then use~\eqref{eq: ind.In} to obtain
\begin{align*}
\|\delta^{m+1}(e^{itD})\| &\leq |t| \int_0^1 \sum_{k=0}^{m} {\textstyle \binom{m}{k}} \|\delta^{m-k}(e^{istD})\|\, \|[\eta I, D]\|\, \|\delta^{k}(e^{i(1-s)tD})\| \, d{s} \\
&\leq  (c\,|t|\,\|[\eta I,D]\|)^{m+1} ce^{{\omega}|t|} \int_0^1 \sum_{k=0}^{m} {\textstyle \binom{m}{k}} {s}^{m-k} (1-{s})^k  \, d{s} \\
&=  (c\,|t|\,\|[\eta I,D]\|)^{m+1} ce^{{\omega}|t|} \int_0^1 ({s} + (1-{s}))^{m}  \, d{s} \\
&=  (c\,|t|\,\|[\eta I,D]\|)^{m+1} ce^{{\omega}|t|}\ .
\end{align*}
 This proves~\eqref{eq: ind.NormIn} for all $n\in\N_0$.

Therefore, using the estimate \eqref{eq: ind.NormIn} in \eqref{easily}, together with property~\eqref{2}, we obtain
\[
|( e^{itD}u,v )| \leq   (c\kappa|t|\|\eta\|_{\Lip})^n ce^{{\omega}|t|} \|u\|_p\|v\|_{p'}
\leq   (c\kappa|t|\alpha)^n ce^{{\omega}|t|} \|u\|_p\|v\|_{p'}
\] for all $n\in\N_0$. We have $ {c} \kappa|t|\alpha<1$, so   $( e^{itD}u,v )=0$ as required.
\end{proof}

 \begin{remark}
In fact we have proved the stronger statement that \[\sppt(e^{itD}u)\subset \tilde K_t(D)\subset K_{c\kappa   |t|},\] where $\tilde K_t(D) =\cap\{\sppt(\eta)\,;\,\eta\in \Lip(M)\,,\,\eta\equiv1   \text{ on }  K\,,\,c|t|\|[\eta I,D]\|_\infty< 1\}.$

For example, if $M=\R^n$  and $\frac{\partial}{\partial x_1}$  does not appear in $D$, then there is no propagation in the $x_1$   direction. See the recent paper of Cowling and Martini \cite{CM} for some related results.\end{remark}

We conclude the section by proving Theorem~\ref{Thm: MainThm1}.

\begin{proof}[Proof of Theorem~\ref{Thm: MainThm1}] 
From what was explained in the introduction, together with the fact that $[\eta I,[\eta I, D]]=0$ when $[\eta I, D]$ is a multiplication operator, Lemma~\ref{Lem: CommForm} also holds under the hypotheses of Theorem \ref{Thm: MainThm1}. We then consider $\epsilon>0$ and repeat the proof of Theorem~\ref{Thm: MainThm2}, replacing the cut-off function ${\eta_{K,\alpha}\in\Lip(M)}$ with a $[0,1]$-valued function $\tilde{\eta}\in C^\infty(M)$ satisfying $\tilde\eta=1$ on~$K$, $\sppt(\tilde\eta)\subset K_{1/\alpha}$, and $\|\nabla\tilde\eta\|_\infty\leq(1+\epsilon)\|\eta_{K,\alpha}\|_{\Lip}$. These approximations exist because the Riemannian manifold $M$ is separable (see, for instance, Corollary~3 in~\cite{AzagraFerreraLopezRangel2007}). The result follows because $\epsilon>0$ can be chosen arbitrarily.
\end{proof}

\begin{remark}
The only reason that the Riemannian manifold $M$ was required to be separable in Theorem~\ref{Thm: MainThm1}, was so that we could construct smooth approximations to Lipschitz functions in the proof above. Indeed, Theorem~\ref{Thm: MainThm2} provides an analogous result without requiring separability.
\end{remark}

\section{Some Special Cases}\label{sec:special}

A typical example of a first order system is the Hodge--Dirac operator $D=\delta+\delta^*$ acting on $L^2(\Lambda(M))$ when $M$ is a complete Riemannian manifold. For this operator, the group $(e^{itD})_{t\in\R}$ has finite propagation speed 1. We shall consider the case when the manifold is $\R^n$ or an open subset thereof, and restrict attention to the leading components of the Hodge--Dirac  operator (the components acting between scalar-valued functions and vector fields), and perturbations thereof. For this purpose, when $\Omega$ is an open subset of $\R^n$, we require the Sobolev space $W^{1,2}(\Omega)$ consisting of all $f$ in $L^2(\Omega)$ with generalised derivatives satisfying
\[
\|f\|_{W^{1,2}(\Omega)}^2:=\|f\|_{L^2(\Omega)}^2 + \|\nabla f\|_{L^2(\Omega)}^2<\infty,
\]
where $\nabla f = (\partial_jf)_{j=1,\ldots,n}$.

\medskip
\noindent \textbf{Case I.} \quad
Let $M=\R^n$ ($n\in\N$), and let  $D$ denote the self-adjoint operator  $$ D=\left[\begin{array}{cc}
0&-\Div \\ \nabla &0 \end{array}\right]:\begin{array}{c}W^{1,2}(\R^n)\\ \oplus\\ \Dom(\Div)\end{array}\subset \begin{array}{c}L^2(\R^n)\\ \oplus\\ L^2(\R^n,\C^n)\end{array}\to \begin{array}{c}L^2(\R^n)\\ \oplus\\ L^2(\R^n,\C^n),\end{array}
$$
where $\nabla: f\mapsto (\partial_jf)_j$ has domain $ W^{1,2}(\R^n)$, and $\Div=-\nabla^* :(u_j)_j\mapsto \sum_j\partial_ju_j$ has  domain  $\{u\in L^2(\R^n,\C^n)\,;\, \Div { u} \in L^2(\R^n)\}$.

It follows from known results that $(e^{itD})_{t\in\R}$ has finite propagation speed 1. It is also a consequence of Theorem~\ref{Thm: MainThm2} with $c=1$ and $\omega =0$ because $D$ is self-adjoint, and with $\kappa=1$ because (using $\|\nabla\eta\|_{\infty} = \|\eta\|_{\Lip}$) we have
$$\left\|[\eta I,D] \left[\begin{array}{c}
f\\(u_j)_j \end{array}\right]\right\|_2 = \left\|\left[\begin{array}{c}
\sum_j (\partial_j\eta) u_j\\-(\partial_j\eta)_j f\end{array}\right]\right\|_2\leq\|\eta\|_{\Lip}\left\|\left[\begin{array}{c}
f\\(u_j)_j \end{array}\right]\right\|_2\ \ \forall\eta\in \Lip(\R^n).
$$

\medskip
\noindent \textbf{Case II.} \quad
Let $D$ denote the operator in Case I, and consider the perturbed operator $BD$ with the same domain, where $B\in L^\infty(\R^n,\cL(\C^{1+n}))$ satisfies $\langle B(x)\zeta,\zeta\rangle \geq \lambda|\zeta|^2$ for a.e.~$x\in\R^n$ and all $\zeta\in\C^{1+n}$, for some $\lambda>0$.
The multiplication operator $B$ is a strictly positive self-adjoint operator in $L^2(\R^n,\C^{1+n})$, since $(Bu,u)\geq\lambda\|u\|^2$. Hence $\|B^{1/2}\|=\|B\|^{1/2}={\|B\|_\infty}^{1/2}$ and $B^{-1}$, $B^{-1/2}$ both exist as bounded operators.

Using these facts, we find that $BD$ is self-adjoint in $L^2(\R^n,\C^{1+n})$ under the  inner product $(u\,,v)_B:=(B^{-1}u\,,v)$, whose associated norm $\|u\|_B=\|B^{-1/2}u\|$ is equivalent to $\|u\|$.  So $iBD$ generates a $C_0$ group $(e^{itBD})_{t\in\R}$ in $L^2(\R^n,\C^{1+n})$ with
$$
\|e^{itBD}u\|\leq{\|B\|_\infty}^{1/2} \|e^{itBD}u\|_B\leq{\|B\|_\infty}^{1/2} \|u\|_B\leq\lambda^{-1/2}{\|B\|_\infty}^{1/2}\|u\|\ .
$$ 
 The commutator $[\eta I, BD]=B[\eta I,D]$ is a multiplication operator which  satisfies $\|[\eta I,BD]u\| \leq \|B\|_\infty\|\eta\|_{\Lip}\|u\|$, so by Theorem~\ref{Thm: MainThm2}, the group $(e^{itBD})_{t\in\R}$ has finite propagation speed $\kappa_{BD}\leq \lambda^{-1/2}{\|B\|_\infty}^{3/2}$. 

Actually, this  can be improved. As noted in Remark \ref{othernorm},   the equivalent norm $\|u\|_B$ on $L^2(\R^n,\C^{1+n})$ can be used in the proof of Theorem \ref{Thm: MainThm2}. The operator $BD$ is self-adjoint in this norm, and  $\|[\eta I,BD]u\|_B =\|B^{-1/2}B[\eta I,D]u\|\leq \|B\|_\infty\|\eta\|_{\Lip}\|u\|_B$, so we conclude that $(e^{itBD})_{t\in\R}$ has finite propagation speed $\kappa_{BD}\leq \|B\|_\infty$.

\medskip
\noindent \textbf{Case III.} \quad Now we allow inhomogeneous terms, and consider an operator of the form $BD$ acting on an open subset $\Omega\subset\R^n$. Suppose that $V$ is a closed subspace of $W^{1,2}(\Omega)$ which contains $C^\infty_c$ (the $C^\infty$ functions with compact support), and which has the property that $\eta f\in V$ for all $\eta\in \Lip(\Omega)$ and $f\in V$. For example, the space $W^{1,2}(\Omega)$ itself has this property, as does $W^{1,2}_0(\Omega)$ (the closure of $C^\infty_c(\Omega)$ in $W^{1,2}(\Omega)$). (This last statement follows from the facts that, given $\eta\in \Lip(\Omega)$, then  $\eta f \in W^{1,2}_0(\Omega)$ for all $f\in C^\infty_c(\Omega)$ and $\eta I \in \cL(W^{1,2}(\Omega))$.) 

Define $\nabla_V:V\subset L^2(\Omega)\to L^2(\Omega,\C^n)$ by $\nabla_Vf=(\partial_j f)_j$, and set ${\Div_V = -{\nabla_V}^*}$. That is, $\Div_Vu=\Div u$ for all $u\in\Dom(\Div_V)=\{u\in L^2(\Omega,\C^n)\,;\, \Div u \in L^2(\Omega)$ and $(-\nabla f\,,u)=(f\,,\Div u) \ \forall f\in V\}$. In particular, we have $\eta u \in \Dom(\Div_V)$ for all $\eta\in \Lip(\Omega)$ and $u \in \Dom(\Div_V)$.

Define the self-adjoint operator
$$ D=\left[\begin{array}{ccc}
0&I&-\Div_V \\I&0&0\\ \nabla_V &0&0 \end{array}\right]:\begin{array}{c}V\\ \oplus\\L^2(\Omega)\\ \oplus\\ \Dom(\Div_V)\end{array}\subset \begin{array}{c}L^2(\Omega)\\ \oplus\\L^2(\Omega)\\ \oplus\\ L^2(\Omega,\C^n)\end{array}\to \begin{array}{c}L^2(\Omega)\\ \oplus\\L^2(\Omega)\\ \oplus\\ L^2(\Omega,\C^n)\end{array}
$$
and the multiplication operator
$$B=\left[\begin{array}{ccc}
a&0&0\\0&A_{00}&(A_{0k})\\0&(A_{j0})&(A_{jk})
\end{array}\right]\in L^\infty(\Omega,\cL(\C^{2+n}))
$$
with $a(x)\geq \lambda>0$ and $\sum_{j,k=1}^nA_{jk}(x)\zeta_k\overline\zeta_j\geq\lambda|\zeta|^2$ for a.e.~$x\in\Omega$ and all $\zeta\in\C^n$, for some $\lambda>0$.

The operator $BD$ satisfies~\eqref{comm} as in Case II. If $A$, and hence $B$, were positive self-adjoint, then $iBD$ would generate a $C_0$ group as before, but we have only assumed this for the matrix-valued function  $(A_{jk})$ with $j,k\in\{1,\ldots,n\}$.
We remedy this by writing 
$$BD = \left[\begin{array}{ccc}
0&a&-a\Div\\
A_{00}+\sum_{k} A_{0k}\partial_k&0&0\\
(A_{j0}+\sum_{k} A_{jk}\partial_k)&0&0 \end{array}\right]=\tilde B D+C\ ,
$$ 
where $\tilde B=\left[\begin{array}{ccc}
a&0&0\\0&\re A_{00} +\alpha&(A_{0k})\\0&({\overline {A_{0j}}})&(A_{jk})
\end{array}\right]$, $C=\left[\begin{array}{ccc}
0&0&0\\i \im A_{00}-\alpha &0&0\\(A_{j0}-\overline {A_{0j}})&0&0
\end{array}\right]$, and $\alpha>0$ is chosen large enough so that $\langle \tilde B(x)\zeta,\zeta\rangle \geq \frac\lambda2|\zeta|^2$ for a.e.$~x\in\Omega$ and all $\zeta\in \C^{2+n}$. As in Case II, we see that $\tilde BD$ is self-adjoint in $L^2(\Omega,\C^{2+n})$ under the inner product $(u\,,v):=(\tilde B^{-1}u\ ,v)$, so $i\tilde BD$ generates a $C_0$ group with $\|e^{it\tilde BD}u\|_{\tilde B}\leq\|u\|_{\tilde B}$. Lemma~\ref{Lem: Group.Pert} below then allows us to deduce that $iBD=i\tilde BD+iC$ generates a $C_0$ group $(e^{itBD})_{t\in\R}$ with 
$$\|e^{itBD}u\|\leq {\|\tilde B\|_\infty}^{1/2}\|e^{it(\tilde BD+C)}u\|_{\tilde B}\leq {\|\tilde B\|_\infty}^{1/2} e^{\tilde{\omega}|t|}\|u\|_{\tilde B}\leq \tilde{c} e^{\tilde{\omega}|t|}\|u\|$$
for some finite $\tilde{c}\geq1$ and $\tilde{\omega}\geq0$. By Theorem~\ref{Thm: MainThm2}, we conclude that $(e^{itBD})_{t\in\R}$ has finite propagation speed.

\begin{lemma}\label{Lem: Group.Pert}
Let $\mathcal{X}$ be a Banach space, and suppose that $T:\Dom(T)\subset \mathcal{X}\rightarrow \mathcal{X}$ is a linear operator that generates a $C_0$ group $(e^{tT})_{t\in\R}$ in $\mathcal{X}$ satisfying $\|e^{tT}\| \leq ce^{\omega |t|}$ for some $c\geq1$ and $\omega\geq0$. If $B\in \cL(\cX)$, then the sum $T+B$ on $\Dom(T)$ generates a $C_0$ group $(e^{t(T+B)})_{t\in\R}$ on $\mathcal{X}$ satisfying  $\|e^{t(T+B)}\| \leq ce^{(\omega + c\|B\|)|t|}$.
\end{lemma}

Lemma~\ref{Lem: Group.Pert} is a well known result based on the work of Phillips in~\cite{Phillips1953}. The proof for semigroups in Theorem III.1.3 of~\cite{EngelNagel2000} can be extended to give the above result.

\section{Weak Huygens' Principle}\label{sec:app}

In this section, we apply Theorem~\ref{Thm: MainThm2} to prove a weak Hugyens' principle for second order hyperbolic equations. For motivation, we start with the wave equation on $\R^n$. A homogeneous equation with bounded measurable coefficients is treated next, followed by an inhomogeneous version on a domain  $\Omega\subset \R^n$.

In Case I, we considered the operator $D=\left[\begin{array}{cc}0&-\Div \\ \nabla &0 \end{array}\right]$ in $\begin{array}{c}L^2(\R^n)\\ \oplus\\ L^2(\R^n,\C^n)\end{array}$, which is self-adjoint, and noted that $iD$ generates a $C_0$ group $(e^{itD})_{t\in\R}$ with finite propagation speed 1. 
Consequently, the cosine family $\cos(tD)= \frac{1}{2}(e^{itD}+e^{-itD})$  ($t\geq0$) also has finite propagation speed 1, where this is defined in the obvious way. Note that the cosine operators, being even functions of $D$, satisfy $\cos(tD)=\cos(t\sqrt{D^2}$), where 
$D^2=\left[\begin{array}{cc}-\Delta&0 \\ 0 &-\nabla \Div \end{array}\right]$ and $\Delta=\Div\nabla$ denotes the Laplacian operator with domain $\Dom(\Delta)=\{f\in W^{1,2}(\R^n)\,;\,\nabla f \in \Dom(\Div)\}$. On restricting attention to the first component, we deduce that the cosine family
$(\cos(t\sqrt{-\Delta}))_{t\geq0}$ has finite propagation speed 1.
This is at the heart of the weak Huygens' principle for the wave equation:

\begin{theorem} If $f \in W^{1,2}(\R^n)$, $g\in L^2(\R^n)$  with $\sppt(f) \cup\sppt(g)\subset K \subset \R^n$, then the solution
 $$F(t) = \cos(t\sqrt{-\Delta}) f+\int_0^t\cos(s\sqrt{-\Delta})g\,d{s} $$ of the Cauchy problem 
 \begin{align*}
 \tfrac{\partial^2}{\partial t^2}F(t) - \Delta F(t) &= 0\qquad (t>0)\\
 \lim_{t\to0}F(t)&=f\\
 \lim_{t\to0} \tfrac{\partial}{\partial t} F(t) &= g
 \end{align*} 
 has support $\sppt(F(t))\subset K_t$.
\end{theorem}
This result is very well known. 
The solution $F$ belongs to 
$C^1(\R^+,L^2(\R^n))\cap C^0(\R^+, W^{1,2}(\R^n))$. There is a considerable literature on the wave equation, so we shall not proceed further  with statements of uniqueness, energy estimates, etc.

We turn now to the  corresponding result for homogeneous equations with $L^\infty$ coefficients.
 Let  $L = - a\Div A \nabla = -a\sum_{j,k=1}^n\partial_j A_{jk}\partial_k$,  where $a\in L^\infty(\R^n)$ 
and  $A=(A_{jk}) \in L^\infty(\R^n,\cL(\C^n))$ with $a(x)\geq \lambda>0$ and $\langle A(x)\zeta,\zeta\rangle\geq\lambda|\zeta|^2$ for a.e.~$x\in\R^n$ and all $\zeta\in\C^n$. Here $L:\Dom(L)\subset L^2(\R^n)\to L^2(\R^n)$ with $\Dom(L)=\{f\in W^{1,2}(\R^n)\,;\,A\nabla f\in \Dom(\Div)\}$.

\begin{theorem}\label{hom} If $f \in W^{1,2}(\R^n)$, $g\in L^2(\R^n)$ with $\sppt(f) \cup\sppt(g)\subset K \subset \R^n$, then the solution
 $$F(t) = \cos(t\sqrt{L}) f+\int_0^t\cos(s\sqrt{L})g\,d s    \in C^1(\R^+,L^2(\R^n))\cap C^0(\R^+, W^{1,2}(\R^n))$$ of the Cauchy problem 
 \begin{align*}
 \tfrac{\partial^2}{\partial t^2}F(t) +L F(t) &= 0\qquad (t>0)\\
  \lim_{t\to0}F(0)&=f\\   \lim_{t\to0}
 \tfrac{\partial}{\partial t} F(0) &= g
 \end{align*} 
 has support $\sppt(F(t))\subset K_{\alpha t}$, where $\alpha:=(\|a\|_\infty\|A\|_\infty)^{1/2}$. \end{theorem}

\begin{proof} Apply Case II with $B= \left[\begin{array}{cc}
 \tfrac a\beta&0\\0&\beta A
\end{array}\right]$ and $\beta=\left(\frac{\|a\|_\infty}{\|A\|_\infty}\right)^{1/2}$. Then \vspace{-.1cm}
$$ BD=\left[\begin{array}{cc}
 \tfrac a\beta&0\\0&\beta A
\end{array}\right] \left[\begin{array}{cc}
0&-\Div \\  \nabla &0 \end{array}\right]  =  \left[\begin{array}{cc}
0&-\tfrac a\beta\Div \\  \beta A\nabla &0 \end{array}\right]\vspace{-.1cm}
$$ and $(BD)^2= \left[\begin{array}{cc}
L &0 \\  0 &\tilde L\end{array}\right] $ with $L$ as above and $\tilde L= -A\nabla a\Div$.
From Case II, the $C_0$ group $(e^{itBD})_{t\in\R}$ has finite propagation speed  $\kappa_{BD} \leq\|B\|_\infty =\alpha$. On defining $\cos(tBD)=\frac12(e^{itBD}+e^{-itBD})$, it is clear that
the cosine family\vspace{-.1cm}
$$(\cos(tBD))_{t\geq0}= (\cos(t\sqrt{(BD)^2}))_{t\geq0}= \left(\left[\begin{array}{cc}
\cos(t\sqrt{L}) &0 \\  0 &\cos(t\sqrt{\tilde L}\,)\end{array}\right] \right)_{t\geq0}$$ has the same bound on its propagation speed. It follows that the first component $(\cos(t\sqrt{L}))_{t\geq0}$, acting on $L^2(\R^n)$,  has the same bound $\alpha$ on its propagation speed,  as required.\end{proof}

\begin{remark}\label{rem:cosine} The operators $\cos(tBD)$ agree with those defined in the standard functional calculus of the self-adjoint operator $BD$ acting on $L^2(\R^n,\C^{1+n})$ with inner product $(u,v)_B = (B^{-1}u,v)$.
Their properties are those of cosine functions as presented, for example, in ~\cite{ArendtBatty,ArendtBattyHieberNeubrander2001}. Note that the use of the square root sign is purely a symbolism to express the fact that the cosine operators $\cos(tBD)$  are diagonal. The function $\cos(t\sqrt{z})$ is an analytic function of $z$. \end{remark}

\begin{remark} When $a=1$, then $L$ is the self-adjoint operator in $L^2(\R^n)$ associated with the sesquilinear form 
$J_A: W^{1,2}(\R^n)\times W^{1,2}(\R^n)\rightarrow \C$  defined by\vspace{-.1cm}
\begin{equation*}\label{eq: homJA}
J_A(f,g) = \int_{\R^n} \sum_{j,k=1}^n A_{jk}(x) (\partial_k f(x)) \partial_j \overline{g}(x)\,dx
\end{equation*}
for all $f$, $g\in W^{1,2}(\R^n)$. See, for example, Chapter IV of~\cite{Kato}. The weak Huygens' principle is  well known for these operators. See, for example, \cite{Sik}. Degenerate elliptic operators are also treated in \cite{Rob}.
\end{remark}

\begin{remark} Our methods work also for systems where the functions are  $\C^N$-valued for some $N \in \N$, and  $A \in L^\infty(\C^{nN})$ satisfies G\aa rding's inequality:
\begin{equation*}\label{eq: homacc}
 J_A(f,f) \geq \lambda \|\nabla f\|_{L^2(\R^n,\C^{nN})}^2\quad\forall f\in W^{1,2}(\R^n,\C^N)\ .
\end{equation*}
The proof that $iBD$ generates a $C_0$ group needs to be modified as follows.

 The positivity condition on $B\in L^\infty(\R^n,\cL(\C^{N+nN}))$ is weakened to $(BDu,\!Du)\!\geq \lambda\|u\|^2$ for all $u\in\Dom(D)$. Then, following \cite{AAMc}, $L^2(\R^n,\C^{N+nN})=\Nul(D)\oplus\overline{\Ran(BD)}$ with corresponding projections $\Proj_{\Nul}$ and $\Proj_{\Ran}$. (Here $\Nul(D)$ denotes the nullspace of $D$ and $\Ran(BD)$ denotes the range of $BD$) The projections, which are typically non-orthogonal, commute with $BD$. Moreover $B:\overline{\Ran(D)}\to \overline{\Ran(BD)}$ has a bounded inverse $B^{-1}:\overline{\Ran(BD)}\to \overline{\Ran(D)}$. The operator  $BD$ is self-adjoint  in $L^2(\R^n,\C^{N+nN})$ under the  inner product $(u\,,v)_B:=(\Proj_\Nul u\,, \Proj_\Nul v)+(B^{-1}\Proj_\Ran u\,,\Proj_\Ran v)$, whose associated norm $\|u\|_B$ is equivalent to $\|u\|$. Thus $iBD$ generates a $C_0$ group. Proceed as in Case II, though note that the bound on $\kappa_{BD}$ is not the same as before.
 \end{remark}

Finally we consider inhomogeneous hyperbolic equations on a domain $\Omega\subset\R^n$. 
As in Case III,  suppose that $V$ is a closed subspace of $W^{1,2}(\Omega)$ which contains 
$C^\infty_c(\Omega)$, and which has the property that $\eta f\in V$ for all $\eta\in\Lip(\Omega)$ and $f\in V$.
Suppose also  that $a, A_{jk}\in L^\infty(\R^n)\ (j,k=0,\dots,n)$ 
 with $a(x)\geq \lambda>0$ and $\sum_{j,k=1}^nA_{jk}\zeta_k\overline\zeta_j\geq\lambda|\zeta|^2$
for a.e.~$x\in\Omega$ and all $\zeta\in\C^n$. 

Define $\nabla_V:V\subset L^2(\Omega)\to L^2(\Omega,\C^n)$ by $\nabla_Vf=(\partial_j f)_j$, and set ${\Div_V = -{\nabla_V}^*}$. That is, $\Div_Vu=\Div u$ for all $u\in\Dom(\Div_V)=\{u\in L^2(\Omega,\C^n)\,;\, \Div u \in L^2(\Omega)$  and  $(-\nabla f\,,u)=(f\,,\Div u) \ \forall f\in V\}$.

Define the operator $L$ in $L^2(\Omega)$  by 
\begin{align*} Lu & =-a\sum_{j,k=1}^n\partial_jA_{jk}\partial_ku  
 -a\sum_{j=1}^n\partial_jA_{j0}u +a\sum_{k=1}^nA_{0k}\partial_ku+aA_{00}u\\
&=a\left[\begin{array}{cc}
I&-\Div_V \end{array}\right]\left[\begin{array}{cc}
A_{00}&(A_{0k})\\(A_{j0})&(A_{jk})
\end{array}\right]\left[\begin{array}{c}
I\\ \nabla_V
\end{array}\right]u
\end{align*}
with
$\Dom(L) = \{u\in V\,;\, (\sum_{k=1}^n A_{jk}\partial_ku + A_{j0}u)_j\in\Dom(\Div_V)\}$.

\begin{theorem}\label{inhom} If $f \in V$, $g\in L^2(\Omega)$ with $\sppt(f) \cup\sppt(g)\subset K \subset \Omega$, then the solution
 \begin{equation}\label{expression} F(t) = \cos(t\sqrt{L}) f+\int_0^t\cos(s\sqrt{L})g\,d s 
 \end{equation} of the Cauchy problem 
 \begin{align*}
 \tfrac{\partial^2}{\partial t^2}F(t) +L F(t) &= 0\qquad (t>0)\\
 \lim_{t\to0} F(0)&=f\\
\lim_{t\to0} \tfrac{\partial}{\partial t} F(0) &= g
 \end{align*} 
 has support $\sppt(F(t))\subset K_{\tilde\alpha t}:=\{x\in \Omega\,;\,\dist(x,K)\leq \tilde\alpha t\}$ for some finite $\tilde\alpha$. \end{theorem}
 
 \begin{proof} With $B$ and $D$ specified as in Case III, we have shown that $iBD$ generates a $C_0$ group $(e^{itBD})_{t\in\R}$ with finite propagation speed that we now denote by $\tilde\alpha$. The cosine family $\cos(tBD) = \frac12(e^{itBD}+e^{-itBD}) = \cos(t\sqrt{(BD)^2})$ ($t\geq0$) has the same propagation speed. Noting that 
 $$(BD)^2 = \left[\begin{array}{ccc}L&0&0\\0&\tilde L_{00}&(\tilde L_{0k})\\0&(\tilde L_{j0})&(\tilde L_{jk})
\end{array} \right] 
 $$ for suitable operators $\tilde L_{jk}$, it follows that the first component $(\cos(t\sqrt{L}))_{t\geq0}$, acting on $L^2(\R^n)$,  has  finite propagation speed bounded by $\tilde\alpha$.  The result follows.
  \end{proof}

\begin{remark}
The use of the square root symbol is again purely symbolic, as $\cos(t\sqrt{L})$ is really just the leading component of $\cos(tBD)$. It is not the case in general that an operator $\sqrt{L}$ is defined. See Remark~\ref{rem:cosine}. In the language of cosine families, it is said that $(BD)^2$, and hence~$L$, are generators of their respective cosine families. In the usual treatment of cosine families associated with $L$, one adds a positive constant to $A_{00}$ if necessary to ensure that $\re J_A(f,f)\geq \|\nabla f\|_2^2 + \|u\|_2^2 $, and notes that the numerical range of $A$ is contained in a parabola. This ensures that $L$ generates a cosine family \cite{Haase}, and that
$\|\sqrt{L}f\|_2\approx \|\nabla f\|_2^2 + \|f\|_2^2 $ \cite{Sqrt}. See also~\cite{ArendtBatty}. Our treatment is consistent with this approach, but we do not need to apply it, as we use the more straightforward fact that, since $iBD$ generates a $C_0$ group, then $(BD)^2$ generates a cosine family, and so then does the restriction to its first component $L$ in $L^2(\R^n)$. 
\end{remark}

\begin{remark}
Given $f\in V$ and $g\in L^2(\R^n)$, the solution $F$ of the Cauchy problem satisfies $F\in  C^1(\R^+,L^2(\R^n))\cap C^0(\R^+, V)$
with
$$\|F(t)\|_2+\|\nabla F(t)\|_2+\|\partial_t F(t)\|_2 \leq c(1+t)e^{\tilde \omega t}\{  \| f\|_2+\|\nabla f\|_2 +\|g\|_2\}
$$ for all $t>0$ and some constant $c$. Let us prove this using what we know about the operator $BD$ in Case III. The constant $\tilde\omega$ is the same one as there.

When $u\in L^2(\R^n,\C^{2+n})$, then $e^{itBD}u \in C^0(\R,L^2(\R^n,\C^{2+n}))$, and when
 $u\in \Dom(D) =\Dom(BD)$, then $e^{itBD}u \in \! C^1(\R,L^2(\R^n,\C^{2+n}))\cap C^0(\R,\Dom(D))$. Hence a similar statement holds for $\cos(tBD)$. On restricting to the first component, we deduce that 
 $\cos(t\sqrt{L})f\in C^1(\R^+,L^2(\R^n))\cap C^0(\R^+,V)$
 with 
 $$\|\cos(t\sqrt{L})f\|_2+\|\nabla\cos(t\sqrt{L})f\|_2+\|\tfrac\partial{\partial t}\cos(t\sqrt{L})f\|_2 \leq  c_1e^{\tilde\omega t}(\|f\|_2+\|\nabla f\|_2)$$ for all $t>0$ and some constant $c_1$.

Next, since $\frac\partial{\partial t}\int_0^t\cos(s\sqrt{L})g\,ds= \cos(t\sqrt{L})g$, we see that $\int_0^t\cos(s\sqrt{L})g\,ds$ is in $C^1(\R^+,L^2(\R^n))$
 with 
 $$\|\int_0^t\cos(s\sqrt{L})g\,ds\|_2 +
 \|\tfrac\partial{\partial t}\int_0^t\cos(s\sqrt{L})g\,ds\|_2\leq \tilde c(1+t)  e^{\tilde\omega t} \|g\|_2\ .$$
 Finally, letting $v:=\left[\begin{array}{c}g\\0\\0\end{array}\right]$, we obtain $\int_0^t\cos(s\sqrt{L})g\,ds \in C^0(\R^+,V)$ with
 \begin{align*}\|\nabla \int_0^t &\cos(s\sqrt L)g\,ds\| _2 \leq \|\tilde B^{-1}\|\|\tilde B D \int_0^t \cos(sBD)v\,ds\| _2\\
 &\leq 2\lambda^{-1}(\| B D \int_0^t \cos(sBD)v\,ds\| _2 +\|C\|\|\int_0^t \cos(sBD)v\,ds\| _2)\\
 &\leq 2\lambda^{-1}(\| \sin(tBD)v\| _2 +\|C\|\tilde cte^{\tilde\omega t}\|v\| _2)
 \leq c_2(1+t)e^{\tilde\omega t}\|g\| _2
 \end{align*}
 for all $t>0$ and some constant $c_2$.
  On using these bounds in \eqref{expression}, we obtain the required estimate. (When $B$ is invertible, then the term $(1+t)$ is not needed.)
\end{remark}

\begin{remark} When $a=1$, then $L$ is the operator in $L^2(\Omega)$ associated with the sesquilinear form 
$J_A: V \times V\rightarrow \C$  defined by
\begin{align*}\label{eq: inhomKA}
J_{A}(f,g) :&= \int_{\Omega} \sum_{j,k=1}^n A_{jk} (\partial_k f) \partial_j \overline{g}
+\sum_{j=1}^n A_{j0}f\,\partial_j\overline{g}
+\sum_{k=1}^n A_{0k}(\partial_kf) \overline{g}
+A_{00}f \overline{g}\\
&=(A\left[\begin{array}{c}I\\ \nabla\end{array}\right]f,\left[\begin{array}{c}I\\ \nabla\end{array}\right]g)
\end{align*}
for all $f$, $g\in V$. See, for instance, Chapter IV in~\cite{Kato}.
\end{remark}

\begin{remark} The choice of $V=W_0^{1,2}(\Omega)$ gives $L$ with Dirichlet boundary conditions, whilst $V=W^{1,2}(\Omega)$ gives $L$ with Neumann boundary conditions, though usually it is assumed in this case that the boundary of $\Omega$ is at least Lipschitz. When $V$ consists of all  functions in $W^{1,2}(\Omega)$ which are zero on part of a Lipschitz  boundary, then the corresponding operator $L$ satisfies mixed boundary value conditions. In \cite{AKMc2}, the authors obtained Davies--Gaffney estimates for the resolvents of such operators under similar hypotheses as here, and in a related fashion.
\end{remark}

\section*{Acknowledgements}\vspace{-.2cm}
The authors acknowledge support from, respectively, the Centre for Mathematics and its Applications (CMA) at the Australian National University, and the Department of Mathematics at the University of Missouri. The authors were also supported by the Australian Government through the Australian Research Council. In particular, the research was conducted during the second author's appointment as a Visiting Fellow at the CMA. The authors thank Tom ter Elst, Andrew Hassell, Sylvie Monniaux and Adam Sikora for helpful conversations and suggestions.

\bibliographystyle{amsplain}

\begin{thebibliography}{10}\vspace{-.2cm}

\bibitem{ArendtBatty}
Wolfgang Arendt, Charles J.~K. Batty,
  \emph{Forms, functional calculus, cosine functions and perturbation}, Perspectives in operator theory, 17--38, Banach Center Publ., \textbf {75}, Polish Acad. Sci., Warsaw, 2007.

\bibitem{ArendtBattyHieberNeubrander2001}
Wolfgang Arendt, Charles J.~K. Batty, Matthias Hieber, Frank Neubrander,
  \emph{Vector-valued {L}aplace {T}ransforms and {C}auchy {P}roblems},
  Monographs in Mathematics, vol.~96, Birkh\"auser Verlag, Basel, 2001.   

\bibitem{AAMc} Pascal Auscher, Andreas Axelsson, Alan McIntosh, On a quadratic estimate related to the Kato conjecture and boundary value problems, {\it Contemp. Math., AMS}, \textbf {505} (2010), 105--129.

\bibitem{AMcN1997}
Pascal Auscher, Alan McIntosh, Andrea Nahmod, \emph{The square root problem
  of {K}ato in one dimension, and first order elliptic systems}, Indiana Univ.
  Math. J. \textbf{46} (1997), no.~3, 659--695.

\bibitem{AKMc}
Andreas Axelsson, Stephen Keith, Alan McIntosh, \emph{Quadratic estimates and functional calculi of perturbed {D}irac
  operators}, Invent. Math. \textbf{163} (2006), no.~3, 455--497.

\bibitem{AKMc2}
Andreas Axelsson, Stephen Keith, Alan McIntosh, \emph{The {K}ato square
  root problem for mixed boundary value problems}, J. London Math. Soc. (2)
  \textbf{74} (2006), no.~1, 113--130.

\bibitem{AzagraFerreraLopezRangel2007}
D.~Azagra, J.~Ferrera, F.~L{\'o}pez-Mesas, Y.~Rangel, \emph{Smooth
  approximation of {L}ipschitz functions on {R}iemannian manifolds}, J. Math.
  Anal. Appl. \textbf{326} (2007), no.~2, 1370--1378.

\bibitem{CMcM}
Andrea Carbonaro, Alan McIntosh, Andrew J. Morris, \emph{Local {H}ardy spaces
  of differential forms on {R}iemannian manifolds}, J. Geom. Anal. DOI 10.1007/s12220-011-9240-x.

\bibitem{CM} Michael Cowling, Alessio Martini, \emph{Sub-Finsler geometry and finite propagation speed}, preprint.

 \bibitem{Rob} A. F. M. ter Elst, Derek Robinson, Adam Sikora, Yueping Zhu, \emph{Second-order operators with degenerate coefficients},
  Proc. London Math. Soc.   \textbf {95} (2007),  299--328.
  
\bibitem{EngelNagel2000}
Klaus-Jochen Engel, Rainer Nagel, \emph{One-parameter {S}emigroups for
  {L}inear {E}volution {E}quations}, Graduate Texts in Mathematics, vol. 194,
  Springer-Verlag, New York, 2000.

\bibitem{Haase}
Markus Haase,
\emph{The group reduction for bounded cosine functions on UMD spaces}, 
Math. Z. \textbf{262} (2009),  281--299. 
  
\bibitem{HigsonRoe2000}
Nigel Higson, John Roe, \emph{Analytic {$K$}-{H}omology}, Oxford Mathematical Monographs, Oxford University Press, Oxford, 2000.

\bibitem{Kato}
Tosio Kato, \emph{{P}erturbation {T}heory for {L}inear {O}perators}, Classics
  in Mathematics, Springer-Verlag, Berlin, 1995.
  
\bibitem{Sqrt} Alan McIntosh  \emph{On representing closed accretive sesquilinear forms as $(A^{1/2}u, A^{*1/2}v)$}. Nonlinear partial differential equations and their applications, Coll\`ege de France Seminar, Vol. III (Paris, 1980/1981), pp. 252--267, Res. Notes in Math., 70.

\bibitem{McIntoshNahmod2000}
Alan McIntosh, Andrea Nahmod, \emph{Heat kernel estimates and functional
  calculi of {$-b\Delta$}}, Math. Scand. \textbf{87} (2000), no.~2, 287--319.

\bibitem{Phillips1953}
R.~S. Phillips, \emph{Perturbation theory for semi-groups of linear operators},
  Trans. Amer. Math. Soc. \textbf{74} (1953), 199--221.

  \bibitem{Sik} Adam Sikora, \emph{Riesz transform, Gaussian bounds and the method of wave equation}, 
Math. Z. \textbf {247} (2004), 643--662.

\vspace{-.2cm}
\end{thebibliography}

\end{document}